\documentclass[11pt,reqno]{amsart}
\usepackage{amssymb}
\usepackage[usenames,dvipsnames]{color}
\usepackage{euscript}
\usepackage{multicol}
\usepackage{amssymb}
\usepackage{amsmath}
\usepackage{times} 

\usepackage[pdftex]{graphicx}
\usepackage{amsmath} 
\usepackage{times} 
 
\def\XXint#1#2#3{{\setbox0=\hbox{$#1{#2#3}{\int}$}
     \vcenter{\hbox{$#2#3$}}\kern-.5\wd0}}

\newtheorem{thm}{Theorem}[section]

\newtheorem{lem}[thm]{Lemma}
\newtheorem{cor}[thm]{Corollary}
\newtheorem{rem}{Remark}

\newcommand{\R}{{\mathbb R}}

\newcommand{\N}{{\mathbb N}}

\numberwithin{equation}{section}

\newcommand{\loglap}{L_{\text{\tiny $\Delta \,$}}\!}

\newcommand{\cD}{{\mathcal D}}

\newcommand{\cH}{{\mathbb H}}

\newcommand{\be}[1]{\begin{equation}\label{#1}}
\newcommand{\ee}{\end{equation}}

\title[Spectral properties of the logarithmic Laplacian]
{Spectral properties of the logarithmic Laplacian}

\author{}

\author{Ari Laptev}
\thanks{Department of Mathematics,
Imperial College London,
Huxley Building, 180 Queen's Gate
London SW7 2AZ, UK, a.laptev@imperial.ac.uk}  

\author{Tobias Weth} 
\thanks{Institut f\"ur Mathematik, Goethe-Universit\"at, Frankfurt, Robert-Mayer-Stra\ss e 10, D-60629 Frankfurt, Germany, weth@math.uni-frankfurt.de.}

\begin{document}

\maketitle

\begin{abstract}
We obtain spectral inequalities and asymptotic formulae for the discrete spectrum of the operator $\frac12\, \log(-\Delta)$ in an open set $\Omega\in\Bbb R^d$, $d\ge2$, of finite measure with Dirichlet boundary conditions. We also derive some results regarding lower bounds for the eigenvalue $\lambda_1(\Omega)$ and compare them with previously known inequalities. 
\end{abstract}

\thispagestyle{empty}
\parindent=0pt
\parskip=5pt

\section{Introduction}
\label{sec:introduction}

In the present paper, we study spectral estimates for the logarithmic Laplacian 
$\loglap= \log (-\Delta)$, which is a (weakly) singular integral operator with Fourier symbol $2\log |\eta|$ and arises as formal derivative $\partial_s \Big|_{s=0} (-\Delta)^s$ of fractional Laplacians at $s= 0$. The study of $\loglap$ has been initiated recently in \cite{HW}, where its relevance for the study of asymptotic spectral properties of the family of fractional Laplacians in the limit $s \to 0^+$ has been discussed. A further motivation for the study of $\loglap$ is given in \cite{jarohs-saldana-weth}, where it has been shown that this operator allows to characterize the $s$-dependence of solution to fractional Poisson problems for the full range of exponents $s \in (0,1)$.  The logarithmic Laplacian also arises in the geometric context of the $0$-fractional perimeter, which has been studied recently in \cite{DNP}. 

For matters of convenience, we state our results for the operator $\mathcal H= \frac{1}{2}\loglap$ 
which corresponds to the quadratic form 
\begin{equation}
\label{log-quadratic}
\varphi \mapsto (\varphi,\varphi)_{log} := \frac{1}{(2\pi)^{d}} \, \int_{\Bbb R^d} \log(|\xi|)\,  |\widehat{\varphi}(\xi)|^2\, d\xi.
\end{equation}
Here and in the following, we let $\widehat{\varphi}$ denote the Fourier transform 
$$
\xi \mapsto \widehat{\varphi}(\xi)= \int_{\R^d} e^{-ix \xi} \varphi(x)\,dx
$$
of a function $\varphi\in L^2(\R^d)$. Let $\Omega\subset \Bbb R^d$ be an open set of finite measure, and let 
$\cH(\Omega)$ denote the closure of $C^\infty_c(\Omega)$ with respect to the norm
\begin{equation}
  \label{eq:def-norm--star}
\varphi \mapsto \|\varphi\|_{*}^2:= \int_{\Bbb R^d} \log(e + |\xi|)\,  |\widehat{\varphi}(\xi)|^2\, d\xi.
\end{equation}
Then $(\cdot,\cdot)_{log}$  defines a closed, symmetric and semibounded quadratic form with domain $\cH(\Omega) \subset L^2(\Omega)$, see Section~\ref{sec:prel-basic-prop} below. Here and in the following, we identify $L^2(\Omega)$ with the space of functions $u \in L^2(\R^d)$ with $u \equiv 0$ on $\R^d \setminus \Omega$. Let 
$$
\mathcal H : \cD(\mathcal H) \subset L^2(\Omega) \to L^2(\Omega) 
$$
be the unique self-adjoint operator associated with the quadratic form. The eigenvalue problem for $\mathcal H$ then writes as 
\begin{equation}\label{D}
\left\{
  \begin{aligned}
\mathcal H \varphi &= \lambda \varphi, &&\qquad \text{in $\Omega$,}\\
\varphi &= 0, &&\qquad \text{on $\R^d \setminus \Omega$.}
  \end{aligned}
\right.
\end{equation}
We understand (\ref{D}) in weak sense, i.e. 
$$
\varphi \in \cH(\Omega) \quad \text{and}\quad (\varphi,\psi)_{log}= \lambda \int_{\Omega}\varphi(x)\psi(x)\,dx \quad \text{for all $\psi \in \cH(\Omega)$.}
$$
As noted in \cite[Theorem 1.4]{HW}, there exists a sequence of eigenvalues 
$$
\lambda_1(\Omega)< \lambda_2(\Omega) \le \dots, \qquad \lim_{k \to \infty} \lambda_k(\Omega) = \infty 
$$
and a corresponding complete orthonormal system of eigenfunctions. We note that the discreteness of the spectrum is a consequence of the fact that the embedding $\cH(\Omega) \hookrightarrow L^2(\Omega)$ is compact. In the case of bounded open sets, the compactness of this embedding follows easily by Pego's criterion~\cite{Pego}. In the case of unbounded open sets of finite measure, the compactness can be deduced from \cite[Theorem 1.2]{jarohs-weth} and estimates for $\|\cdot\|_*$, see Corollary~\ref{cor-compact-embedding} below. 

In Section \ref{sec:prel-basic-prop}, using the results from \cite{HW} and \cite{FKV}, we discuss properties  of functions from  $\mathcal D(\mathcal H)$. In particular, we show that $e^{ix\xi}\big|_{x\in\Omega} \in \mathcal D(\mathcal H)$, $\xi\in\Bbb R^d$, provided $\Omega$ is an open bounded sets with Lipschitz boundary.

In Section \ref{sec:deriving-an-upper-1} we obtain a sharp upper bound for the Riesz means and for the number of eigenvalues  $N(\lambda)$ of the operator $\mathcal H$ below $\lambda$. Here we use technique developed in papers \cite{Bz1}, \cite{Bz2}, \cite{LY} and \cite{L}. In \cite{Lap} it was noticed that such technique could be applied for a class of pseudo-differential operators with Dirichlet boundary conditions in domains of finite measure without any requirements on the smoothness of the boundary. 

 We discuss  lower bounds for $\lambda_1(\Omega)$ in Section \ref{sec:lower-bound-lambd}. In Theorem \ref{lower-bound-lambda_1-first} we present an estimate that is valid for arbitrary open sets of finite measure. For sets with Lipschitz boundaries, H.Chen and T.Weth \cite{HW} have proved a Faber-Krahn inequality for the operator $\mathcal H$ that reduces the problem to the estimate of 
$\lambda_1(B)$, where $B$ is a ball satisfying $|B| = |\Omega|$, see Corollary \ref{cor-faber-krahn}. In Theorem \ref{lower-bound-lambda-1-second} we find an estimate for $\lambda_1(B_d)$, where $B_d$ is the unit ball, that is better in lower dimensions than the one obtained in Theorem \ref{lower-bound-lambda_1-first}. We also compare our results with bounds resulting from previously known spectral inequalities obtained in  \cite{BK} and \cite{B}.

In Section \ref{LowB1} we obtain asymptotic lower bounds using the coherent states transformation approach given in \cite{G}. It allows us to derive, in Section \ref{Weyl}, asymptotics for the Riesz means of eigenvalues in  Theorem \ref{3.1} and for $N(\lambda)$ in Corollary \ref{3.2}. Here $\Omega \subset \R^N$ is an arbitrary open set of finite measure without any additional restrictions on the boundary. 

Finally in Section \ref{LowB2} we obtain uniform bounds on the Riesz means of the eigenvalues using the fact that for bounded open sets with Lipschitz boundaries we have $e^{ix\xi}\big|_{x\in\Omega} \in \mathcal D(\mathcal H)$.


\section{Preliminaries and basic properties of eigenvalues}
\label{sec:prel-basic-prop}

As before, let $(\cdot,\cdot)_{log}$ denote the quadratic form defined in (\ref{log-quadratic}), and let, for an open set $\Omega \subset \R^d$, 
 $\cH(\Omega)$ denote the closure of $C^\infty_c(\Omega)$ with respect to the norm $\|\cdot\|_*$ defined in (\ref{eq:def-norm--star}). 

\begin{lem}
\label{closed-semibounded}
Let $\Omega \subset \R^d$ be an open set of finite measure. Then $(\cdot,\cdot)_{log}$  defines a closed, symmetric and semibounded quadratic form with domain $\cH(\Omega) \subset L^2(\Omega)$.
\end{lem}

\begin{proof}
Obviously, the form $(\cdot,\cdot)_{log}$ is  symmetric. For functions $\varphi \in C^\infty_c(\Omega)$, we have 
\begin{equation}
  \label{eq:basic-fourier-ineq}
(2\pi)^{d} \|\varphi\|_2^2 =\|\widehat \varphi\|_2^2 \le \|\varphi\|_*^2. 
\end{equation}
Moreover, with $c_1:= \log (e+2)+ \sup \limits_{t \ge 2}\frac{\log (e+t)}{\log t}$
we have
\begin{align}
\frac{\|\varphi\|_{*}^2}{c_1} &\le \| \widehat \varphi\|_{2}^2+  \int_{|\xi| \ge 2}\ln |\xi| |\widehat{\varphi}(\xi)|^2  \, d\xi \nonumber\\
&\le (2\pi)^d \bigl( \|\varphi\|_{2}^2 +  (\varphi,\varphi)_{log}\bigr) 
- \int_{|\xi| \le 2}\ln |\xi| |\widehat{\varphi}(\xi)|^2  \, d\xi \nonumber \\
&\le (2\pi)^d \bigl( \|\varphi\|_{2}^2 +  (\varphi,\varphi)_{log}\bigr) 
+ \bigl\| \ln |\cdot| \bigr\|_{L^1(B_2(0))} \|\widehat{\varphi}\|_\infty^2 \label{closed-semibounded-est-1}
\end{align}
while 
\begin{equation}
  \label{eq:closed-semibounded-est-2}
\|\widehat{\varphi}\|_\infty^2 \le \|\varphi\|_1^2 \le |\Omega|\, \|\varphi\|_2^2.
\end{equation}
Consequently, 
\begin{align}
(\varphi,\varphi)_{log} &\ge   \frac{\|\varphi\|_*^2}{(2\pi)^d c_1}-
\left(1+ \frac{|\Omega|\, \bigl\| \ln |\cdot| \bigr\|_{L^1(B_2(0))}}{(2\pi)^{d}}\right)\|\varphi\|_2^2 \label{intermediate-est}\\
&\ge  \left(\frac{1}{c_1}\,-\,1\,-\,\frac{|\Omega|\, \bigl\| \ln |\cdot| \bigr\|_{L^1(B_2(0))}}{(2\pi)^{d}}\right)\|\varphi\|_2^2. \nonumber
\end{align}
In particular, $(\varphi,\varphi)_{log}$ is semibounded. Moreover, it follows from (\ref{intermediate-est}) and the completeness of $(\cH(\Omega),\|\cdot\|_*)$ that the form $(\varphi,\varphi)_{log}$ is closed on $\cH(\Omega)$.  
\end{proof}

\begin{lem}
\label{equivalent-norms}
Let $\Omega \subset \R^d$ be an open set of finite measure. Then 
\begin{equation}
  \label{eq:def-norm-double-star}
\varphi \mapsto \|\varphi\|_{**}^2:= 
\int \!\!\! \int_{|x-y|\le 1} \frac{(\varphi(x)-\varphi(y))^2}{|x-y|^d}\,dxdy
\end{equation}
defines an equivalent norm to the norm $\|\cdot\|_*$ defined in (\ref{eq:def-norm--star}) on $C^\infty_c(\Omega)$. 
\end{lem}

\begin{proof}
Let $\varphi \in C^\infty_c(\Omega)$. By \cite[Lemma 2.7]{FKV}, we have 
\begin{equation}
  \label{eq:FKV-lemma}
\|\varphi\|_2 \le c_2 \|\varphi\|_{**}  \qquad \text{with a constant $c_2>0$ independent of $\varphi$.}  
\end{equation}
In particular, $\|\cdot\|_{**}$ defines a norm on $C^\infty_c(\Omega)$.
Next we note that, by \cite[Theorem 1.1(ii) and Eq. (3.1)]{HW},   
$$
(\varphi,\varphi)_{log} = \frac{1}{2}\int_{\R^d}[\loglap \varphi(x)]\varphi(x)\,dx = \kappa_d \|\varphi\|_{**}^2 - \int_{\Bbb R^d} [j * \varphi] \varphi \,dx + \zeta_d  \|\varphi\|_2^2
$$
with 
\begin{equation}
  \label{eq:def-zeta_d}
\kappa_d:= \frac{\pi^{- \frac{d}{2}}  \Gamma(d/2)}{4}, \qquad \zeta_d:= 
\log 2 + \frac{1}{2}\left(\psi\left(d/2\right) -\gamma\right)
\end{equation}
and 
$$
j: \R^d \setminus \{0\} \to \R, \qquad j(z)= 2 \kappa_d 1_{\R^d \setminus B_d}(z)|z|^{-d}.
$$
Here $\psi:= \frac{\Gamma'}{\Gamma}$ is the Digamma function and $\gamma= -\Gamma'(1)$ is the Euler-Mascheroni constant. 
Consequently, we have 
\begin{align}
\Bigl|(\varphi,\varphi)_{log}- \kappa_d \|\varphi\|_{**}^2\Bigr| &\le 
\|j\|_\infty \|\varphi\|_1^2 + \zeta_d  \|\varphi\|_2^2 \nonumber\\
&\le \Bigl(\|j\|_\infty |\Omega| +\zeta_d\Bigr)\|\varphi\|_2^2.  \label{modulus-ineq-quad-form}
\end{align}
As a consequence of (\ref{eq:basic-fourier-ineq}) and (\ref{modulus-ineq-quad-form}), we find that 
\begin{align*}
\|\varphi\|_{**}^2 &\le \frac{1}{\kappa_d}\Bigl[ (\varphi,\varphi)_{log} + 
\bigl(\|j\|_\infty |\Omega|+\zeta_d\bigr) \|\varphi\|_2^2 \Bigr]\\
&\le \frac{1}{(2\pi)^d \kappa_d}\Bigl(1  + \|j\|_\infty |\Omega|+\zeta_d\Bigr)\|\varphi\|_*^2.
\end{align*}
Moreover, by (\ref{closed-semibounded-est-1}), (\ref{eq:closed-semibounded-est-2}), (\ref{eq:FKV-lemma}) and (\ref{modulus-ineq-quad-form}) we have 
\begin{align*}
&\frac{\|\varphi\|_{*}^2}{c_1}  
\le (2\pi)^d \bigl(\|\varphi\|_{2}^2 + (\varphi,\varphi)_{log}\bigr) 
+   \bigl\| \ln |\cdot| \bigr\|_{L^1(B_2(0))} |\Omega| \|\varphi\|_2^2\\
&\le (2\pi)^d \Bigl( \kappa_d \|\varphi\|_{**}^2
+ \bigl(1+ \|j\|_\infty |\Omega| +\zeta_d\bigr)\|\varphi\|_2^2 \Bigr) 
+   \bigl\| \ln |\cdot| \bigr\|_{L^1(B_2(0))} |\Omega| \|\varphi\|_2^2\\
&\le c_3 \|\varphi\|_{**}^2
\end{align*}
with $c_3 = (2\pi)^d \kappa_d + c_2\bigl[(2\pi)^d \bigl(1+ \|j\|_\infty |\Omega|+\zeta_d\bigr) +\bigl\| \ln |\cdot| \bigr\|_{L^1(B_2(0))}|\Omega| \bigr]$. Hence the norms $\|\cdot\|_{*}$ and $\|\cdot\|_{**}$ are equivalent on $C^\infty_c(\Omega)$.
\end{proof}

\begin{cor}
\label{cor-compact-embedding}
Let $\Omega \subset \R^d$ be an open set of finite measure. Then the embedding $\cH(\Omega) \hookrightarrow L^2(\Omega)$ is compact. 
\end{cor}

\begin{proof}
Let $\tilde \cH(\Omega)$ be defined as the space of functions $\varphi \in L^2(\R^d)$ with $\varphi \equiv 0$ on $\R^d \setminus \Omega$ and 
$$
\int \!\!\! \int_{|x-y|\le 1} \frac{(\varphi(x)-\varphi(y))^2}{|x-y|^d}\,dxdy <\infty.
$$  
By \cite[Theorem 1.2]{jarohs-weth}, the Hilbert space $(\tilde \cH(\Omega),\|\cdot\|_{**})$ is compactly embedded in $L^2(\Omega)$. 
Since, by Lemma~\ref{equivalent-norms}, the norms $\|\cdot\|_*$ and $\|\cdot\|_{**}$ are equivalent on $C^\infty_c(\Omega)$, the space $\cH(\Omega)$ is embedded in $\tilde \cH(\Omega)$. Hence the claim follows.    
\end{proof}

\begin{cor}
\label{space-equivalence}
Let $\Omega \subset \R^d$ be a bounded open set with Lipschitz boundary.
\begin{enumerate}
\item[(i)] The space $\cH(\Omega)$ is equivalently given as the set of functions $\varphi \in L^2(\R^d)$ with $\varphi \equiv 0$ on $\R^d \setminus \Omega$ and 
  \begin{equation}
    \label{eq:kernel-finiteness-cond}
\int \!\!\! \int_{|x-y|\le 1} \frac{(\varphi(x)-\varphi(y))^2}{|x-y|^d}\,dxdy <\infty.
  \end{equation}
\item[(ii)] $\cH(\Omega)$ contains the characteristic function $1_\Omega$ of $\Omega$ and also the restrictions of exponentials $x \mapsto 1_\Omega(x) \, e^{ix \xi}$, $\xi \in \R^d$.
\end{enumerate}
\end{cor}

\begin{proof}
(i) Let, as in the proof of Corollary~\ref{cor-compact-embedding}, $\tilde \cH(\Omega)$ be the space of functions $\varphi \in L^2(\R^d)$ with $\varphi \equiv 0$ on $\R^d \setminus \Omega$ and with (\ref{eq:kernel-finiteness-cond}), endowed with the norm $\|\cdot\|_{**}$. Since $\Omega \subset \R^d$ be a bounded open set with Lipschitz boundary, it follows from \cite[Theorem 3.1]{HW} that $C_0^\infty(\Omega) \subset \tilde \cH(\Omega)$ is dense. Hence the claim follows from Lemma~\ref{equivalent-norms}.

(ii) follows from (i) and a straightforward computation.   
\end{proof}

Next we note an observation regarding the scaling properties of the eigenvalues $\lambda_k(\Omega)$. 
\begin{lem}
\label{lemma-scaling-properties}  
Let $\Omega \subset \R^d$ be a bounded open set with Lipschitz boundary, and let 
$$
R\Omega:= \{R x\::\: x \in \Omega\}. 
$$
Then we have
$$
\lambda_k(R \Omega) = \lambda_k(\Omega) -  \log R \qquad \text{for all $k \in \N$.}
$$
\end{lem}

\begin{proof}
Since $C_0^\infty(\Omega) \subset \cH(\Omega)$ is dense, it suffices to note that
\begin{equation}
  \label{eq:scaling-test-functions}
(\varphi_R,\varphi_R)_{log} = (\varphi,\varphi)_{log} - \log R \|\varphi\|_{L^2(\R^d)}^2 \qquad \text{for $\varphi \in C^\infty_c(\R^d)$}
\end{equation}
with $\varphi_R \in C^\infty_c(\R^d)$ defined by $\varphi_R(x)= R^{-\frac{d}{2}}\varphi(\frac{x}{R})$, whereas $\|\varphi_R\|_{L^2(\R^d)}= \|\varphi\|_{L^2(\R^d)}$. Since 
$$
\widehat{\varphi_R}= R^{\frac{d}{2}} \widehat{\varphi}(R \,\cdot \,)
$$
we have 
\begin{align*}
&(\varphi_R,\varphi_R)_{log}\\
&= \frac{1}{(2\pi)^{d}} \, \int_{\Bbb R^d} \log(|\xi|)\,  |\widehat{\varphi_R}(\xi)|^2\, d\xi = \frac{R^{d}}{(2\pi)^{d}} \, \int_{\Bbb R^d} \log(|\xi|)\,  |\widehat{\varphi}(R \xi)|^2\, d\xi\\
&= \frac{1}{(2\pi)^{d}} \, \int_{\Bbb R^d} \bigl(\log(|\xi|)-\log R\bigr)\,  |\widehat{\varphi}(\xi)|^2\, d\xi =(\varphi,\varphi)_{log} - \log R \|\varphi\|_{L^2(\R^d)}^2,
\end{align*}
as stated in (\ref{eq:scaling-test-functions}).\\ 
\end{proof}
 
\section{An upper trace bound}
\label{sec:deriving-an-upper-1}

Throughout this section, we let $\Omega \subset \R^d$ denote an open set of finite measure. Let $\{\varphi_k\}$ and $\{\lambda_k\}$ be the orthonormal in $L^2(\Omega)$ system of eigenfunctions and the eigenvalues of the operator $\mathcal H$  respectively. In what follows we denote
$$
(\lambda - t)_+ = 
\begin{cases}
\lambda - t, & {\rm if} \quad t <\lambda, \\
0, \quad & {\rm if} \quad  t \ge \lambda.
\end{cases}
$$
Then we have
\begin {thm}\label{1.1}
For the eigenvalues of the problem \eqref{D} and any $\lambda\in \Bbb R$ we have
\begin{equation}\label{BU}
\sum_{k}(\lambda - \lambda_k)_+ \le \frac{1}{(2\pi)^{d}}\, |\Omega|\, e^{d\lambda} \, |B_d|\, d^{-1},
\end{equation}
where $|B_d|$ is the measure of the unit ball in $\Bbb R^d$.
\end{thm}

\begin{proof}
Extending the eigenfunction $\varphi_k$ by zero outside $\Omega$ and using the Fourier transform we find
\begin{multline*}
\sum_{k}(\lambda - \lambda_k)_+ = \sum_{k}\left(\lambda (\varphi_k, \varphi_k) - (\mathcal H\varphi_k, \varphi_k) \right)_+ \\
= \frac{1}{(2\pi)^{d}}\, \left(\sum_k \int_{\Bbb R^d} \left(\lambda - \log(|\xi|) \right) \, |\widehat{\varphi_k}(\xi)|^2 \, d\xi \right)_+\\
\le
 \frac{1}{(2\pi)^{d}}\, \int_{\Bbb R^d} \left(\lambda - \log(|\xi|) \right)_+ \,  \sum_k |\widehat{\varphi_k}(\xi)|^2 \, d\xi. 
\end{multline*} 
Using that $\{\varphi_k\}$ is an orthonormal basis in $L^2(\Omega)$ and denoting 
$e_\xi = e^{-i (\cdot,\xi)}$we have 
$$
\sum_k |\widehat{\varphi_k}(\xi)|^2 = \sum_k |(e_\xi, \varphi_k)|^2 = \|e_\xi\|^2_{L^2(\Omega)} = |\Omega|,
$$
and finally obtain
\begin{align*}
\sum_{k}(\lambda - \lambda_k)_+ & \le  \frac{1}{(2\pi)^{d}}\, |\Omega|\, \int_{\Bbb R^d} \left(\lambda - \log(|\xi|) \right)_+ \\
& =  \frac{1}{(2\pi)^{d}}\, |\Omega|\, e^{d\lambda} \,  \int_{|\xi|\le 1} \log(|\xi|^{-1}) \, d\xi.
\end{align*}
We complete the proof by computing the last integral. 
\end{proof}

\noindent
Let $\eta >\lambda$ and let us consider the function
$$
\psi_\lambda(t) = \frac{1}{\eta - \lambda} (\eta - t)_+.
$$
Denote by $\chi$ the step function 
$$
\chi_\lambda (t) = 
\begin{cases} 
1,  \quad & {\rm if} \quad t<\lambda,\\
0,\quad & {\rm if} \quad t \ge \lambda,
\end{cases}
$$
and let  
$$
N(\lambda) = \# \{k:\, \lambda_k<\lambda\},
$$
be the number of the eigenvalues below $\lambda$ of the operator $\mathcal H$.

Then by using the previous statement we have 
$$
N(\lambda) \le \frac{1}{\eta - \lambda} \, \sum_k (\eta - \lambda_k)_+ \le 
\frac{1}{\eta - \lambda} \,  \frac{1}{(2\pi)^{d}}\, |\Omega|\, e^{d\eta} \,  |B_d|\, d^{-1}.
$$
Minimising the right hand side w.r.t. $\eta$ we find $\eta = \lambda + \frac1d
$ and thus obtain the following

\begin{cor}
\label{cor-N-lambda}
For the number $N(\lambda)$ of the eigenvalues of the operator $\mathcal H$ below $\lambda$ we have
\begin{equation}\label{Numb}
N(\lambda) \le 
 e^{\lambda d +1} 
\frac{1}{(2\pi)^{d}}\, |\Omega| \, |B_d|.
\end{equation}
\end{cor}


\section{A lower bound for $\lambda_1(\Omega)$}
\label{sec:lower-bound-lambd}

In this section, we focus on lower bounds for the first eigenvalue $\lambda_1= \lambda_1(\Omega)$. From Corollary~\ref{cor-N-lambda}, we readily deduce the following bound. 

\begin{thm} 
\label{lower-bound-lambda_1-first}
Let $\Omega \subset \R^d$ be an open set of finite measure. Then we have 
\begin{equation}
  \label{eq:est-lambda_1-first}
\lambda_1(\Omega) \ge \frac{1}{d} \log \frac{(2\pi)^{d}}{e |\Omega| \, |B_d|}.
\end{equation}
In particular, if $|\Omega| \le \frac{(2\pi)^{d}}{e\, |B_d|}$, then the operator $\mathcal H$ does not have negative eigenvalues.
\end{thm}

\begin{proof}
If $\lambda < \frac{1}{d} \log \frac{(2\pi)^{d}}{e |\Omega| \, |B_d|}$, then 
$N(\lambda)<1$ by (\ref{Numb}), and therefore $N(\lambda)=0$. Consequently, $\mathcal H$ does not have eigenvalues below $\frac{1}{d} \log \frac{(2\pi)^{d}}{e |\Omega| \, |B_d|}$.
\end{proof}

\begin{rem}
Note that the inequalities \eqref{BU}, \eqref{Numb} and \eqref{eq:est-lambda_1-first} hold for any open set $\Omega$ of finite measure without any additional conditions on its boundary.
\end{rem}


In the following, we wish to improve the bound given in Theorem~\ref{lower-bound-lambda_1-first} in low dimensions $d$ for open boundary sets with Lipschitz boundary. We shall use the following Faber-Krahn type inequality.

\begin{thm} (\cite[Corollary 1.6]{HW})\\
\label{sec:faber-Krahn-main}
Let $\rho>0$. Among all bounded open sets $\Omega$ with Lipschitz boundary and $|\Omega| = \rho$, the ball $B=B_r(0)$ with $|B|=\rho$ minimizes $\lambda_1(\Omega)$.
\end{thm}
 
\begin{cor}
\label{cor-faber-krahn}
For every open bounded sets $\Omega$ with Lipschitz boundary we have 
\begin{equation}
  \label{eq:sharp-lower-bound}
\lambda_1(\Omega) \ge \lambda_1(B_d) + \frac{1}{d}\log \frac{|B_d|}{|\Omega|},
\end{equation}
and equality holds if $\Omega$ is a ball.   
\end{cor}

\begin{proof}
The result follows by combining Theorem~\ref{sec:faber-Krahn-main} with the identity 
$$
\lambda_1(B_r(0)) = \lambda_1(B_d) + \log \frac{1}{r}\qquad \text{for $r>0$,}
$$
which follows from the scaling property of $\lambda_1$ noted in Lemma~\ref{lemma-scaling-properties}.
\end{proof}

Corollary~\ref{cor-faber-krahn} gives a sharp lower bound, but it contains the  
unknown quantity $\lambda_1(B_d)$. By Theorem~\ref{lower-bound-lambda_1-first}, we have 
\begin{align}
\lambda_1(B_d) &\ge \frac{1}{d} \log \frac{(2\pi)^{d}}{e |B_d|^2} = 
\log (2\pi) -\frac{1}{d}\bigl(1+ 2 \log |B_d|\bigr) \nonumber \\
&= \frac{2}{d} \log \Gamma\left(d/2\right)  + \log 2 + \frac{2}{d} \log \frac{d}{2} -\frac{1}{d}.
\label{lower-bound-lambda-1-first}  
\end{align}
The following theorem improves this lower bound in low dimensions $d \ge 2$. 

\begin{thm} 
\label{lower-bound-lambda-1-second}
For $d \ge 2$, we have 
\begin{equation}
  \label{eq:est-lambda-1-first}
\lambda_1(B_d) \ge \log \bigl(2 \sqrt{d +2}\bigr) -  
\frac{2^{d+1} |B_d|^2 (d +2)^{\frac{d}2} }{d (2\pi)^{2d}}.
\end{equation}
\end{thm}

\begin{proof}
 Let $u \in L^2(B_d)$ be radial with $\|u\|_{L^2}=1$. Then $\widehat u$ is also radial, and 
\begin{align*}
|\widehat u(\xi)|&=|\widehat u(s)|= s^{1-\frac{d}{2}}\left|\int_{0}^{1} u(r)J_{\frac{d}{2}-1}(rs)r^{\frac{d}{2}}dr\right| \\
  &\le s^{1-\frac{d}{2}} \left( \int_0^{1}r^{d-1} u^2(r)\,dr\right)^{1/2} 
\left(\int_0^{1}rJ_{\frac{d}{2}-1}^2(sr) \,dr\right)^{1/2}\\ 
  &=\frac{s^{1-\frac{d}{2}}}{\sqrt{|S^{d-1}|}} 
\left(s^{-2} \int_0^{s} \tau J_{\frac{d}{2}-1}^2(\tau) \,d\tau\right)^{1/2} \\
&=\frac{s^{-\frac{d}{2}}}{\sqrt{|S^{d-1}|}} 
\left(\int_0^{s} \tau J_{\frac{d}{2}-1}^2(\tau) \,d\tau\right)^{1/2}\qquad \text{for $\xi \in \R^d$ with $s = |\xi|$.}
\end{align*}
Consequently, 
$$
|S^{d-1}| |\widehat u(s)|^2 \le s^{-d} \int_0^{s} \tau J_{\frac{d}{2}-1}^2(\tau) \,d\tau.
$$
In the case where, in addition, $u$ is a radial eigenfunction of (\ref{D}) corresponding to $\lambda_1$ in $\Omega= B_d$, it follows that, for every $\lambda \in \R$,  
\begin{align*}
&(2\pi)^{d}[\lambda-\lambda_1] =  \int_{\R^d} (\lambda -\ln |\xi|)|\widehat u(\xi)|^2\,d\xi \le  \int_{\R^d} (\lambda -\ln |\xi|)_+|\widehat u(\xi)|^2\,d\xi\\
&=|S^{d-1}| 
\int_0^\infty s^{d-1} (\lambda -\ln s)_+|\widehat u(s)|^2\,ds \le \int_0^\infty \frac{(\lambda -\ln s)_+}{s}  \int_0^{s} \!\!\!\tau J_{\frac{d}{2}-1}^2(\tau) \,d\tau \,ds\\
&= \int_0^\infty  \tau J_{\frac{d}{2}-1}^2(\tau) \int_{\tau}^\infty \frac{(\lambda -\ln s)_+}{s}  \,ds d\tau= \int_0^{e^\lambda}  \tau J_{\frac{d}{2}-1}^2(\tau) \int_{\tau}^{e^\lambda} \frac{\lambda -\ln s}{s} \,ds d\tau\\
&= \int_0^{e^\lambda}  \tau J_{\frac{d}{2}-1}^2(\tau) \int_{\ln \tau}^{\lambda}(\lambda - s) \,ds d\tau= \int_0^{e^\lambda}  \tau J_{\frac{d}{2}-1}^2(\tau) \int_{0}^{\lambda- \ln \tau} s \,ds d\tau\\
&= \frac{1}{2}  \int_0^{e^\lambda}  \tau J_{\frac{d}{2}-1}^2(\tau)
\bigl(\lambda- \ln \tau \bigr)^2 \,d\tau = \frac{e^{2\lambda}}{2}  \int_0^{1}  \tau J_{\frac{d}{2}-1}^2(e^\lambda \tau) \ln^2 \tau \,d\tau. 
\end{align*}
We now use the following estimate for Bessel functions of the first kind:
\begin{equation}
  \label{eq:bessel-est-proof}
J_\nu(x) \le \frac{x^\nu}{2^\nu \Gamma(\nu+1)} \quad \text{for $\:\nu > \sqrt{3}-2$, $\:0 \le x < 2 \sqrt{2(\nu+2)}$.}
\end{equation}
A proof of this elementary estimate is given in the Appendix. We wish to apply (\ref{eq:bessel-est-proof}) with $\nu = \frac{d}{2}-1$. This gives 
$$
e^{2\lambda} J_{\frac{d}{2}-1}^2(r_0 e^\lambda \tau) \le e^{d\lambda} \frac{\tau^{d-2}}{2^{d-2}\Gamma^2 (\frac{d}{2})}=
 \frac{d^2 |B_d|^2e^{d\lambda}}{(2\pi)^{d}} \tau^{d-2}
 \qquad \text{for $\tau \in [0,1]$} 
$$ 
if $d \ge 2$ and $e^\lambda \le 2 \sqrt{d +2}$, i.e., if   
\begin{equation}
\label{condition-proof}
d \ge 2 \quad \text{and}\quad \lambda \le \log \bigl(2 \sqrt{d +2}\bigr).
\end{equation}
Here we used that $|B_d|= \frac{2}{d} \frac{\pi^{\frac{d}{2}}}{\Gamma(d/2)}$. Consequently, if (\ref{condition-proof}) holds, we find that  
$$
(2\pi)^{d}[\lambda-\lambda_1] \le  \frac{d^2 |B_d|^2e^{d\lambda}}{(2\pi)^{d}}\int_0^{1} 
 \tau^{d-1} \ln^2 \tau \,d\tau,
$$
where 
$$
\int_0^{1} \tau^{d-1} \ln^2 \tau d\tau = - \frac{2}{d}
\int_0^1 \tau^{d-1} \ln \tau d\tau = \frac{2}{d^2} \int_{0}^1 \tau^{d-1}d\tau
= \frac{2}{d^3}.
$$
Hence 
$$
(2\pi)^{d}[\lambda -\lambda_1] \le \frac{2|B_d|^2}{d (2\pi)^{d}}e^{d\lambda}, \quad \text{i.e.,}\quad \lambda_1 \ge \lambda- \frac{2|B_d|^2 }{d (2\pi)^{2d}} e^{d\lambda}.
$$
Inserting the value $\lambda = \log \bigl(2 \sqrt{d +2}\bigr)$ from (\ref{condition-proof}), we deduce that 
$$
\lambda_1= \lambda_1(B_d) \ge  \log \bigl(2 \sqrt{d +2}\bigr) - \frac{2^{d+1} |B_d|^2 (d +2)^{\frac{d}2} }{d(2\pi)^{2d}},  
$$
as claimed. 
\end{proof}

\begin{rem}{\rm 
\label{rem-comparison-of-other-bounds}
It seems instructive to compare the lower bounds given in (\ref{lower-bound-lambda-1-first}) and (\ref{eq:est-lambda-1-first}) with other bounds obtained from spectral estimates which are already available in the literature. We first mention Beckner's logarithmic estimate of uncertainty \cite[Theorem 1]{B}, which implies that\footnote{We note here that a different definition of Fourier transform is used in \cite{B} and therefore the inequality looks slightly different}
\begin{equation*}
(\varphi,\varphi)_{log} \ge \int_{\R^d} \left[\psi\left(d/4\right)+ \log \frac{2}{|x|}\right]\varphi^2(x) dx \ge \left[\psi\left(d/4\right)+ \log 2\right]\|\varphi\|_2^2 
\end{equation*}
for functions $\varphi \in C^\infty_c(B_d)$ and therefore 
\begin{equation}
  \label{eq:beckner-lambda-1-est}
\lambda_1(B_d) \ge \psi\left(d/4\right)+ \log 2 .
\end{equation}
Here, as before, $\psi = \frac{\Gamma'}{\Gamma}$ denotes the Digamma function. Next we state a further lower bound for $(\varphi,\varphi)_{log}$ which follows from \cite[Proposition 3.2 and Lemma 4.11]{HW}. We have
\begin{equation} 
\label{cw-inequality}
(\varphi,\varphi)_{log} \ge \zeta_d \|\varphi\|_2^2 \qquad \text{for $\varphi \in C^\infty_c(B_d)$,}
\end{equation}
where $\zeta_d$ is given in (\ref{eq:def-zeta_d}), i.e., 
$$
\zeta_d =  \log 2 + \frac{1}{2}\left( \psi(d/2)-\gamma\right) = \left\{
  \begin{aligned}
  &- \gamma + \sum_{k=1}^{\frac{d-1}{2}} \frac{1}{2k-1},&&\qquad \text{$d$ odd,}\\
  &\log 2 - \gamma + \sum_{k=1}^{\frac{d-2}{2}} \frac{1}{k},&&\qquad \text{$d$ even.}   
  \end{aligned}
\right.
$$
Inequality (\ref{cw-inequality}) implies that
\begin{equation} 
\label{cw-lambda-1-bound}
\lambda_1(B_d) \ge \zeta_d.
\end{equation}
The latter inequality can also be derived from a lower bound of Ba$\rm{\tilde{n}}$uelos and Kulczycki for the 
first Dirichlet eigenvalue $\lambda_1^\alpha(B_d)$ of the fractional Laplacian $(-\Delta)^{\alpha/2}$ in $B_d$. In 
\cite[Corollary 2.2]{BK}, it is proved that 
$$
\lambda_1^\alpha(B_d) \ge 2^\alpha \frac{\Gamma(1+\frac{\alpha}{2}) \Gamma(\frac{d+\alpha}{2})}{\Gamma(\frac{d}{2})}\qquad \text{for $\alpha \in (0,2)$.}
$$
Combining this inequality with the characterization of $\lambda_1(B_d)$ given in \cite[Theorem 1.5]{HW}, we deduce that   
$$
\lambda_1(B_d)= \lim_{\alpha \to 0^+}\frac{\lambda_1^\alpha(B_d)-1}{\alpha}\ge \frac{d}{d\alpha}\Big|_{\alpha=0}\, 2^\alpha \frac{\Gamma(1+\frac{\alpha}{2}) \Gamma(\frac{d+\alpha}{2})}{\Gamma(\frac{d}{2})} = \zeta_d, 
$$
as stated in (\ref{cw-lambda-1-bound}). 

We briefly comment on the quality of the lower bounds obtained here in low and high dimensions. In low dimensions $d \ge 2$, (\ref{eq:est-lambda-1-first}) is better than the bounds (\ref{lower-bound-lambda-1-first}), (\ref{eq:beckner-lambda-1-est}) and (\ref{cw-lambda-1-bound}). In dimension $d=1$ where the bound (\ref{eq:est-lambda-1-first}) is not available, the bound (\ref{lower-bound-lambda-1-first}) yields the best value. The following table shows numerical values of the bounds $b_1(d)$, $b_2(d)$, $b_3(d)$ resp. $b_4(d)$ given by (\ref{lower-bound-lambda-1-first}), (\ref{eq:est-lambda-1-first}), (\ref{eq:beckner-lambda-1-est}), (\ref{cw-lambda-1-bound}), respectively. 
\medskip

\begin{center}
{\tiny
\renewcommand{\arraystretch}{1.6}
    \begin{tabular}{ l | l | l | l | l |l | l | l | l | l | l |}
    $d$ & 1 & 2 & 3 & 4&5&6&7&8&9&10\\ \hline
    $b_1(d)$ & $-0,55$ & $0,19$ & $0,55$ & $0,79$ &$0,97$&$1,12$ & $1,25$ & $1,36$ & $1,46$ & $1,55$
\\ \hline
$b_2(d)$ & $\quad/$& $1,28 $& $1,48 $ & $1,59 $& $1,67$&$1,73$ &$1,79$ & $1,84$& $1,89$ & $1,94$\\ \hline
$b_3(d)$ &$-3.53$ & $-1,27$ & $-0,39$ &$0,12$&$0,47$ & $0,73$&$0,94$& $1,12$ &$1,27$ & $1,40$ \\ \hline
$b_4(d)$ &$-0,58$& $0,12$ &$0,42$ &$0,62$& $0,76$ &$0,87$ &$0,96$ & $1,03$ & $1,10$ & $1,16$
\end{tabular}
\renewcommand{\arraystretch}{1}
}
\end{center}
\medskip

To compare the bounds in high dimensions, we consider the asymptotics as $d \to \infty$. 
Since $\frac{\log \Gamma(t)}{t} = \log t - 1 + o(t)$ as $t \to \infty$, the bound (\ref{lower-bound-lambda-1-first}) yields 
\begin{equation}
\label{lower-bound-lambda-1-first-asymptotics}
\lambda_1(B_d) \ge \log d - 1 + o(1) \qquad \text{as $d \to \infty$,} 
\end{equation}
whereas (\ref{eq:est-lambda-1-first}) obviously gives 
\begin{equation}
\label{lower-bound-est-lambda-1-first-asymptotics}
\lambda_1(B_d) \ge \log \sqrt{d+2} + \log 2  + o(1) \qquad \text{as $d \to \infty$,} 
\end{equation}
Moreover, from (\ref{eq:beckner-lambda-1-est}) and the fact that 
\begin{equation}
  \label{eq:Digamma-asymptotics}
\psi(t) = \log t + o(1)\qquad\text{as $t \to \infty$,}   
\end{equation}
we deduce that 
\begin{equation}
  \label{eq:beckner-lambda-1-est-asymptotics}
\lambda_1(B_d) \ge \log d - \log 2 + o(1) \qquad \text{as $d \to \infty$,} 
\end{equation}
Finally, (\ref{cw-inequality}) and (\ref{eq:Digamma-asymptotics}) yield
\begin{equation} 
\label{cw-lambda-1-bound-asymptotics}
\lambda_1(B_d) \ge \log \sqrt{d} +  \log 2  -\frac{\gamma}{2} + o(1) \qquad \text{as $d \to \infty$.} 
\end{equation}
So (\ref{eq:beckner-lambda-1-est-asymptotics}) provides the best asymptotic bound as $d \to \infty$.

Numerical computations indicate that the bound (\ref{eq:est-lambda-1-first}) is better than the other bounds for $2 \le d \le 21$, and (\ref{eq:beckner-lambda-1-est}) is the best among these bounds for $d \ge 22$.
}

\end{rem}


\section{An asymptotic lower trace bound}\label{LowB1}

Throughout this section, we let $\Omega \subset \R^d$ denote an open set of finite measure. In this section we prove the following asymptotic lower bound. A similar statement was obtained in \cite{G} for the Dirichlet boundary problem for a fractional Laplacian.

\begin {thm}\label{2.1}
For the eigenvalues of the problem \eqref{D} and any $\lambda\in \Bbb R$ we have
\begin{equation}\label{BLow}
\liminf_{\lambda\to\infty} e^{-d\lambda} \sum_{k}(\lambda - \lambda_k)_+ \ge \frac{1}{(2\pi)^{d}}\, |\Omega|\, \, |B_d|\, d^{-1}.
\end{equation}
\end{thm}

\begin{proof}
Let us fix $\delta>0$ and consider 
$$
\Omega_\delta = \{ x\in \Omega: \, {\rm dist}(x, \Bbb R^d \setminus \Omega) >\delta\}.
$$
Since $\delta$ is arbitrary it  suffices to show the lower bound \eqref{BLow}, where $\Omega$ is replaced by $\Omega_\delta$.
Let $g\in C_0^\infty(\Bbb R^d)$ be a real-valued even function, $\|g\|_{L^2(\Bbb R^d)} = 1$ with support in $\{x\in \Bbb R^d: \, |x| \le \delta/2\}$. For $\xi\in \Bbb R^d$ and $x\in \Omega_\delta$ we introduce the \lq\lq coherent state" 
$$
e_{\xi,y}(x) =  e^{-i\xi x} g(x-y).
$$
Note that $\|e_{\xi,y}\|_{L^2(\Bbb R^d)} = 1$. Using the properties of coherent states \cite[Theorem 12.8]{LL} we obtain
$$
\sum_{k}(\lambda - \lambda_k)_+ \ge 
 \frac{1}{(2\pi)^{d}}\, \int_{\Bbb R^d} \int_{\Omega_\delta} (e_{\xi,y}, (\lambda - \mathcal H)_+ e_{\xi,y})_{L^2(\Omega)} \, dy d\xi.
$$
Since $t \mapsto (\lambda-t)_+$ is convex then applying Jensen's inequality to the spectral measure of $\mathcal H$ we obtain
\begin{equation}\label{jensen}
\sum_{k}(\lambda - \lambda_k)_+ \ge   \frac{1}{(2\pi)^{d}}\, \int_{\Bbb R^d} \int_{\Omega_\delta}
\left(\lambda - (\mathcal H e_{\xi,y}, e_{\xi,y})_{L^2(\Omega)} \right)_+ \, dy d\xi.
\end{equation}
Next we consider the quadratic form 
\begin{multline*}
\left(\mathcal H e_{\xi,y},  e_{\xi,y} \right)_{L^2(\Omega)} = \frac{1}{(2\pi)^d} \, \int_{\Bbb R^d} \int_{\Omega} \int_{\Omega} e^{i(x-z)(\eta-\xi)} g(x-y)g(z-y) \log(|\eta|)  \, dz dx d\eta \\
=
\frac{1}{(2\pi)^d} \, \int_{\Bbb R^{d}} \int_{\Omega} \int_{\Omega} 
e^{i(x-z)\rho} g(x-y)g(z-y) \log(|\xi-\rho|) \, dz dx d\rho\\
= \frac{1}{(2\pi)^d} \, \int_{\Bbb R^{d}} \int_{\Omega} \int_{\Omega} 
e^{i(x-z)\rho} g(x-y)g(z-y) \left( \log|\xi| + \log \left(\left|\xi -\rho\right|/|\xi| \right)\right)
\, dz dx d\rho\\
=  \log|\xi| + R(y,\xi).
\end{multline*}
Since $g\in C_0^\infty(\Bbb R^d)$ we have for any $M>0$
\begin{multline*}
R(y,\xi) =  \\
\frac{1}{(2\pi)^d} \, \int_{\Bbb R^{d}} \int_{\Omega} \int_{\Omega} 
e^{i(x-y)\rho} g(x-y) e^{i(y-z)\rho} g(z-y)  \log \left(\left|\xi -\rho\right|/|\xi| \right)
\, dz dx d\rho\\
= \int_{\Bbb R^{d}} |\widehat{g}|^2\, \log \left(\left|\xi -\rho\right|/|\xi| \right) \, d\rho
 \le C_M\,
\int_{\Bbb R^{d}} (1+ |\rho|)^{-M}   \log \left(\left|\xi -\rho\right|/|\xi| \right)  \,  d\rho\\
\le C\, |\xi|^{-1}.
\end{multline*}
Therefore from \eqref{jensen} we find 
\begin{equation}\label{below1}
\sum_{k}(\lambda - \lambda_k)_+  \ge (2\pi)^{-d}\, |\Omega_\delta| \, \int_{\Bbb R^d}  (\lambda -  \log|\xi| -  C|\xi|^{-1})_+ \, d\xi.
\end{equation}
Let us redefine the spectral parameter $\lambda =  \ln \mu$.
Then introducing polar coordinates we find
\begin{multline}\label{below22} 
\int_{\Bbb R^d}  (\lambda -  \log|\xi| -  C|\xi|^{-1})_+ \, d\xi = \left|\Bbb S^{d-1} \right| \, \int_0^\infty \left(\ln \frac{\mu}{r} - \frac{C}{r}\right)_+ \, r^{d-1}dr\\
= 
\mu^d\, \left|\Bbb S^{d-1} \right| \, \int_0^\infty \left(\ln \frac{1}{r} - \frac{C}{\mu r}\right)_+ \, r^{d-1}dr
\end{multline} 
The expression in the latter integral is positive if $ - r\ln r > C\mu^{-1}$. The function $ -r\ln r $ is concave. 

\smallskip

\qquad\qquad\qquad \qquad \qquad{\centering 
{\includegraphics[scale=.4]{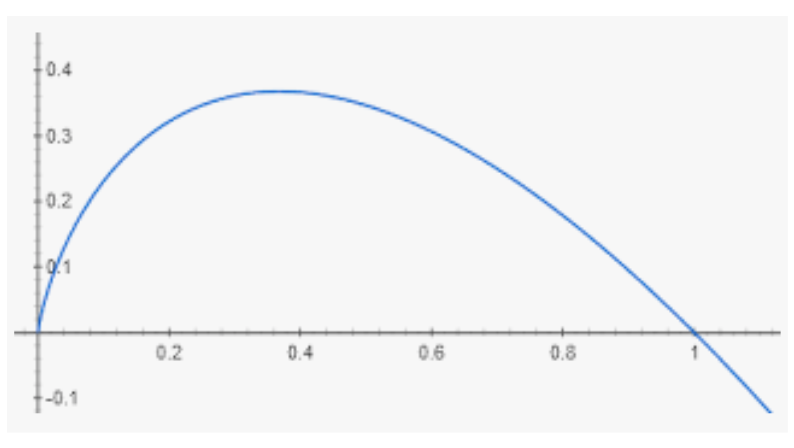}} }
 
\smallskip
 \noindent
Its maximum is achieved at $r=1/e$ at the value $1/e$. The equation
$ - r\ln r = C\mu^{-1}$ has two solutions $r_1(\mu)$ and $r_2(\mu)$ such that $r_1(\mu) \to 0$ and $r_2(\mu)\to 1$ as $\mu \to\infty$
Therefore 
\begin{multline}\label{below3}
\int_0^\infty \left(\ln \frac{1}{r} - \frac{C}{\mu r}\right)_+ \, r^{d-1}dr \ge 
\int_{r_1(\mu)}^{r_2(\mu)}     \left(\ln \frac{1}{r} - \frac{C}{\mu r}\right)  \, r^{d-1}dr \\
=-\frac1d\, r^d \ln r  \Big|_{r_1(\mu)}^{r_2(\mu)}+ \frac{C}{\mu(d+1)} r^{d+1}  \Big|_{r_1(\mu)}^{r_2(\mu)}
+ \frac{1}{d^2}r^d   \Big|_{r_1(\mu)}^{r_2(\mu)} \to \frac{1}{d^2} \quad {\rm as} \quad \mu\to\infty.
\end{multline} 
Putting together \eqref{below1}, \eqref{below22} and \eqref{below3} and using $\mu = e^\lambda$ we obtain
$$
\liminf_{\lambda\to\infty} e^{-d\lambda} \sum_{k}(\lambda - \lambda_k)_+ \ge \frac{1}{(2\pi)^{d}}\, |\Omega_\delta|\, \, |B_d|\, d^{-1}.
$$
Since $\delta>0$ is arbitrary we complete the proof of  Theorem \ref{2.1}.

\end{proof}

\section{Weyl asymptotics}\label{Weyl}

\noindent
Throughout this section, we let $\Omega \subset \R^d$ denote an open set of finite measure. Combining Theorems \ref{1.1} and \ref{2.1} we have


\begin {thm}\label{3.1}
The Riesz means of the eigenvalues of the Dirichlet boundary value problem  \eqref{D} satisfy the following asymptotic formula 
\begin{equation}\label{Weyl1}
 \lim_{\lambda\to\infty} e^{-d\lambda}\, \sum_{k}(\lambda - \lambda_k)_+ = \frac{1}{(2\pi)^{d}}\, |\Omega|\,  |B_d|\, d^{-1}.
\end{equation}
\end{thm}
As a corollary we can obtain asymptotics of the number of  the eigenvalues of the operator $\mathcal H$.


\begin{cor} \label{3.2}
The number of the eigenvalues $N(\lambda)$ of the Dirichlet boundary value problem  \eqref{D} below $\lambda$ satisfies the following asymptotic formula 
\begin{equation}\label{Weyl2}
\lim_{\lambda\to\infty} e^{-d\lambda} \, N(\lambda) =   \frac{1}{(2\pi)^{d}}\, |\Omega|\,  |B_d|.
\end{equation}
\end{cor} 

\begin{proof}
In order to prove \eqref{Weyl2} we use two simple inequalities. If $h>0$, then
\begin{equation}\label{Nabove}
\frac{(\lambda + h - \lambda_k)_+   - (\lambda - \lambda_k)_+}{h} \ge 1_{\text{\tiny $(-\infty,\lambda)$}}(\lambda_k)
\end{equation}
and
\begin{equation}\label{Nbelow}
\frac{(\lambda - \lambda_k)_+   - (\lambda - h- \lambda_k)_+}{h}
 \le 1_{\text{\tiny $(-\infty,\lambda)$}}(\lambda_k)
\end{equation}

The inequality \eqref{Nabove} implies, together with Theorems \ref{1.1} and~\ref{2.1}, that 
\begin{align*}
&\limsup_{\lambda \to \infty}e^{-d\lambda}N(\lambda) \le 
\limsup_{\lambda \to \infty}e^{-d\lambda} \sum_{k}\frac{(\lambda + h - \lambda_k)_+   - (\lambda - \lambda_k)_+}{h}\\
&\le \frac{1}{h}\Bigl[e^{dh} \limsup_{\lambda \to \infty}e^{-d(\lambda+h)} \sum_{k}(\lambda + h - \lambda_k)_+   -\liminf_{\lambda \to \infty}e^{-d\lambda} \sum_{k}(\lambda - \lambda_k)_+\Bigr]\\
&\le \frac{|\Omega| |B_d|}{d(2\pi)^d}\:\frac{e^{dh}-1}{h} \qquad \text{for every $h>0$}
\end{align*}
and thus 
\begin{equation}
  \label{eq:limsup-N-lambda-ineq}
\limsup_{\lambda\to\infty}e^{-d\lambda}N(\lambda) \le \frac{|\Omega| |B_d|}{d(2\pi)^d}\lim_{h \to 0^+}\frac{e^{dh}-1}{h}= \frac{|\Omega|\,  |B_d|}{(2\pi)^{d}}.
\end{equation}
Moreover, \eqref{Nabove} implies, together with Theorems \ref{1.1} and~\ref{2.1}, that 
\begin{align*}
&\liminf_{\lambda \to \infty}e^{-d\lambda}N(\lambda) \ge 
\liminf_{\lambda \to \infty}e^{-d\lambda} \sum_{k}\frac{(\lambda - \lambda_k)_+   - (\lambda -h - \lambda_k)_+}{h}\\
&\ge \frac{1}{h}\Bigl[e^{dh} \liminf_{\lambda \to \infty}e^{-d \lambda} \sum_{k}(\lambda - \lambda_k)_+   -e^{-dh}\limsup_{\lambda \to \infty} e^{-d(\lambda-h)} \sum_{k}(\lambda -h - \lambda_k)_+\Bigr]\\
&\ge \frac{|\Omega| |B_d|}{d(2\pi)^d}\:\frac{1-e^{-dh}}{h} \qquad \text{for every $h>0$}
\end{align*}
and therefore 
\begin{equation}
  \label{eq:liminf-N-lambda-ineq}
\liminf_{\lambda\to\infty}e^{-d\lambda}N(\lambda) \ge \frac{|\Omega| |B_d|}{d(2\pi)^d}\lim_{h \to 0^+}\frac{1-e^{-dh}}{h}= \frac{|\Omega|\,  |B_d|}{(2\pi)^{d}}.
\end{equation}
The claim follows by combining (\ref{eq:limsup-N-lambda-ineq}) and (\ref{eq:liminf-N-lambda-ineq}).  
\end{proof}

\section{An exact lower trace bound}\label{LowB2}

In this section we prove the following exact lower bound in the case of bounded open sets with Lipschitz boundary.

\begin {thm}\label{2.1-new-lower-bound}
Let $\Omega \subset \R^d$, $N \ge 2$ be an open bounded set with Lipschitz boundary, let $\tau \in (0,1)$, and let 
\begin{equation}
\label{def-C-Omega}
C_{\Omega,\tau} := \frac{1}{|\Omega|(2\pi)^d} \int_{\Bbb R^d}(1+|\rho|)^\tau \log(1+|\rho|) |\widehat{1_\Omega}(\rho)|^2\,d\rho, 
\end{equation}
where $1_\Omega$ denotes the indicator function of $\Omega$. 

For any $\lambda \ge 2 C_{\Omega,\tau}$,  we have
\begin{equation}\label{BL}
\sum_{k}(\lambda - \lambda_k)_+
\ge \frac{|\Omega|\, |B_d|}{(2\pi)^{d}\,d} \Bigl[e^{d\lambda} \,- \,a_\tau\, C_{\Omega,\tau}\,e^{(d-\tau)\lambda} \,- \,b_\tau\, C_{\Omega,\tau}^2\, 
e^{(d-2\tau)\lambda} \,-\, (d \lambda + 1)  \Bigr] \nonumber
\end{equation}
with $a_\tau:= \frac{d(d-\tau)-1}{d-\tau}$ and $b_\tau := 4d \tau$.
\end{thm}

\begin{rem}
In the definition of $C_{\Omega,\tau}$, we need $\tau<1$, otherwise the integral might not converge. In particular, if $\Omega=B_d$ is the unit ball in $\R^d$, we have 
$$
\widehat{1_\Omega}(\rho)= (2\pi)^{\frac{d}{2}} |\rho|^{-\frac{d}{2}}J_{\frac{d}{2}}(|\rho|)
$$
where $J_{\frac{d}{2}}(r)= O(\frac{1}{\sqrt{r}})$ as $r \to \infty$. Hence the integral defining $C_{\Omega,\tau}$ converges if $\tau <1$. A similar conclusion arises for cubes or rectangles, where 
$$
\widehat{1_\Omega}(\rho) = f_1(\rho_1) \cdot \dots \cdot f_d(\rho_d)
$$
and $f_j(s) = O(\frac{1}{s})$ as $|s| \to \infty$, $j=1,\dots,d$.

On the other hand, if $\Omega \subset \R^d$ is an open bounded set with Lipschitz boundary, we have 
\begin{equation}
  \label{eq:C-Omega-tau-finiteness}
C_{\Omega,\tau}<\infty \qquad \text{for $\tau \in (0,1)$.}  
\end{equation}
Indeed, in this case, $\Omega$ has finite perimeter, i.e., $1_\Omega \in BV(\R^d)$. Therefore, as noted e.g. in \cite[Theorem 2.14]{Lombardini}, $\Omega$ also has finite fractional perimeter
$$
P_\tau(\Omega)= \int_{\Omega}\int_{\R^d \setminus \Omega} |x-y|^{-d-\tau}\,dxdy = \frac{1}{2} \int \!\! \int_{\R^{2d}}\frac{(1_\Omega(x)-1_\Omega(y))^2}{|x-y|^{d+\tau}}\,dxdy
$$
for every $\tau \in (0,1)$. Moreover, $P_\tau(\Omega)$ coincides, up to a constant, with the integral
$$
\int_{\Bbb R^d}|\rho|^\tau |\widehat{1_\Omega}(\rho)|^2\,d\rho
$$
which therefore is also finite for every $\tau \in (0,1)$. Since moreover $1_\Omega$ and therefore also $\widehat{1_\Omega}$ are functions in $L^2(\Bbb R^d)$ and for every $\varepsilon>0$ there exists $C_\varepsilon>0$ with 
$$
(1+|\rho|)^\tau \log(1+|\rho|) \le C_\varepsilon \bigl(1 + |\rho|^{\tau+\varepsilon}\bigr) \qquad \text{for $\rho \in \R^d$,}
$$
it follows that (\ref{eq:C-Omega-tau-finiteness}) holds. 
\end{rem}

In the proof of Theorem~\ref{2.1-new-lower-bound}, we will use the following elementary estimate. 

\begin{lem}
\label{elem-lemma}
For $r \ge 0$, $s>0$ and $\tau \in (0,1)$, we have 
\begin{equation}
  \label{eq:first-elem-ineq}
\log\left(1 + \frac{r}{s}\right) \le \frac{1}{s} \log(1+r) \qquad \text{if $s \in (0,1)$}
\end{equation}
and 
\begin{equation}
  \label{eq:second-elem-ineq}
\log\left(1 + \frac{r}{s}\right) \le \frac{(1+r)^\tau}{s^\tau} \log(1+r) \qquad \text{if $s \ge 1$.}
\end{equation}
In particular, 
$$
\log\left(1 + \frac{r}{s}\right) \le \max \left\{\frac{1}{s}, \frac{1}{s^\tau} \right\} 
(1+r)^\tau\log(1+r) \qquad \text{for $r,s>0$.}
$$
\end{lem}

\begin{rem}
The obvious bound $\log(1 + \frac{r}{s}) \le \frac{r}{s}$ will not be enough for our purposes. We need an upper bound of the form $g(s)h(r)$ where $h$ grows less than linearly in $r$.    
\end{rem}

\begin{proof}[Proof of Lemma~\ref{elem-lemma}]
Let first $s \in (0,1)$. Since 
$$
\log\left(1 + \frac{r}{s}\right)\Big|_{r=0} = 0 = \frac{1}{s} \log(1+r)\Big|_{r=0}
$$
and, for every $r>0$, 
$$
\frac{d}{dr} \log\left(1 + \frac{r}{s}\right) = 
\frac{1}{s+r} \le \frac{1}{s+ sr} =  \frac{d}{dr} \frac{1}{s} \log(1+r),
$$
inequality (\ref{eq:first-elem-ineq}) follows. To see (\ref{eq:second-elem-ineq}), we fix $s>1$, and we note that
$$
\log\left(1 + \frac{r}{s}\right)\Big|_{r=0} = 0 = \frac{(1+r)^\tau}{s^\tau} \log(1+r)\Big|_{r=0}.
$$
Moreover, for $0 < r \le s-1$, we have  
\begin{align*}
&\frac{d}{dr} \frac{(1+r)^\tau}{s^\tau} \log(1+r)= 
\frac{(1+r)^{\tau-1}}{s^\tau}(1+ \tau \log(1+r))\\
&\ge \frac{(1+r)^{\tau-1}}{s^\tau} 
\ge \frac{1}{s}\ge \frac{1}{s+r}= \frac{d}{dr} \log\left(1 + \frac{r}{s}\right),
\end{align*}
so the inequality holds for $r \le s-1$. If, on the other hand, $r \ge s-1$, we have obviously
$$
\log\left(1 + \frac{r}{s}\right) \le \log(1 + r) \le \frac{(1+r)^\tau}{s^\tau} \log(1+r).
$$
\end{proof}

We may now complete the 

\begin{proof}[Proof of Theorem~\ref{2.1-new-lower-bound}]
For $\xi \in \R^d$, we define $f_\xi \in L^2(\Bbb R^d)$ by $f_\xi(x)= \frac{1}{\sqrt{|\Omega|}}1_{\Omega} e^{-i x \xi}$. Note that  $\|f_{\xi}\|_{L^2(\Bbb R^d)} = 1$ for any $\xi \in \Bbb R^d$. We write
\begin{align*}
\sum_{k}(\lambda - \lambda_k)_+ &= \sum_{k}(\lambda - \lambda_k)_+ \|\varphi_k\|_{L^2(\Omega)}^2 
= \frac{1}{(2\pi)^{d}} \sum_{k}(\lambda - \lambda_k)_+ 
\|\widehat{\varphi_k}\|_{L^2(\Bbb R^{d})}^2 
\\
&=\frac{|\Omega|}{(2\pi)^{d}}  \sum_{k}(\lambda - \lambda_k)_+ \int_{\Bbb R^{d}} |\langle f_\xi,\varphi_k \rangle|^2\,d\xi
\\
&=\frac{|\Omega|}{(2\pi)^{d}}  \int_{\Bbb R^{d}} \sum_{k}(\lambda - \lambda_k)_+  |\langle f_\xi,\varphi_k \rangle|^2\,d\xi.
\end{align*}
Since $\sum \limits_{k} |\langle f_\xi,\varphi_k \rangle|^2 = \|f_{\xi}\|_{L^2(\Bbb R^d)}^2 = 1$ for $\xi \in \Bbb R^d$, Jensen's inequality gives 
\begin{align}
\sum_{k}(\lambda - \lambda_k)_+ &\ge 
\frac{|\Omega|}{(2\pi)^{d}}  \int_{\Bbb R^{d}} \Bigl( \lambda \sum_{k}|\langle f_\xi,\varphi_k \rangle|^2    - \sum_k \lambda_k  |\langle f_\xi,\varphi_k \rangle|^2\Bigr)_+\,d\xi \nonumber\\
&=\frac{|\Omega|}{(2\pi)^{d}}  \int_{\Bbb R^{d}} \Bigl( \lambda   - \sum_k \lambda_k  |\langle f_\xi,\varphi_k \rangle|^2\Bigr)_+\,d\xi \nonumber\\
&=\frac{|\Omega|}{(2\pi)^{d}}  \int_{\Bbb R^{d}} \Bigl( \lambda   - ( \mathcal H  f_\xi, f_\xi ) \Bigr)_+\,d\xi. \label{jensen-new}
\end{align}
Here, since 
$$
\sqrt{|\Omega|}\, \widehat {f_\xi}(\eta-\xi)= 
\int_{\Omega}e^{-i(\eta-\xi) x}e^{-ix \xi}\,dx = 
\int_{\Omega}e^{-i \eta x} \,dx = \widehat{1_\Omega}(\eta)
$$
for $\eta,  \xi \in \R^d$, we have 
\begin{align}
&|\Omega|(2\pi)^d \left(\mathcal H  f_{\xi}, f_{\xi}    \right) = |\Omega| \int_{\Bbb R^d}
\log |\eta| |\widehat {f_\xi}(\eta)|^2d \eta = |\Omega| \int_{\Bbb R^d}
\log |\eta-\xi| |\widehat {f_\xi}(\eta-\xi)|^2d \eta \nonumber\\
&=\int_{\Bbb R^d} \log|\eta-\xi| |\widehat{1_\Omega}(\eta)|^2\,d\eta \le \int_{\Bbb R^d}
\left[ \log |\xi| +\log (1+|\eta|/|\xi|)\right] |\widehat{1_\Omega}(\eta)|^2\,d\eta \nonumber\\
&\le \log |\xi| \int_{\Bbb R^d} |\widehat{1_\Omega}(\eta)|^2\,d\eta+ \max \left\{ \frac{1}{|\xi|}, \frac{1}{|\xi|^\tau}\right\}
\int_{\Bbb R^d}(1+|\eta|)^{\tau} (\log(1+|\eta|)|\widehat{1_\Omega}(\eta)|^2\,d\eta
\nonumber\\
&= |\Omega|(2\pi)^d \Bigl(\log |\xi| +  \max \left\{ \frac{1}{|\xi|}, \frac{1}{|\xi|^\tau}\right\} C_{\Omega,\tau} \Bigr)\qquad \text{for $\xi \in \R^d$,}\label{jensen-new-compl} 
\end{align}
where $C_{\Omega,\tau}$ is defined in (\ref{def-C-Omega}). Here we used Lemma~\ref{elem-lemma}. Combining (\ref{jensen-new}) and (\ref{jensen-new-compl}), we get
\begin{equation}
\sum_{k}(\lambda - \lambda_k)_+ \ge 
\frac{|\Omega|}{(2\pi)^{d}}  \int_{\Bbb R^{d}} \Bigl(\lambda - \log |\xi| - 
\max \left\{ \frac{1}{|\xi|}, \frac{1}{|\xi|^\tau}\right\} C_{\Omega,\tau}\Bigr)_+d\xi. \label{intermediate-new} 
\end{equation}
Let us redefine the spectral parameter $\lambda =  \log \mu$ again.
Then we find
\begin{align}
&\int_{\Bbb R^{d}} \Bigl(\lambda - \log |\xi| - 
\max \left\{ \frac{1}{|\xi|}, \frac{1}{|\xi|^\tau} \right\} C_{\Omega,\tau} \Bigr)_+d\xi \nonumber
\\
&= \left|\Bbb S^{d-1} \right| \, \int_0^\infty \left(\log \frac{\mu}{r} - 
\max \left\{ \frac{1}{r}, \frac{1}{r^\tau}\right\} C_{\Omega,\tau}
\right)_+ \, r^{d-1}dr  \nonumber \\
&= 
\mu^d\, \left|\Bbb S^{d-1} \right| \, \int_0^\infty \left(-\log r  - 
\max \left\{ \frac{1}{\mu^{1-\tau} r}, \frac{1}{r^\tau}\right\} \frac{C_{\Omega,\tau}}{\mu^\tau}
\right)_+ \, r^{d-1}dr
 \nonumber \\
&\ge 
\mu^d\, \left|\Bbb S^{d-1} \right| \, \int_{\frac{1}{\mu}}^\infty \left(-\log r  - 
\frac{1}{r^\tau}\frac{C_{\Omega,\tau}}{\mu^\tau}
\right)_+ \, r^{d-1}dr.  \label{below2}
\end{align}
For the last inequality, we used the fact that $\frac{1}{\mu^{1-\tau} r} \le 
\frac{1}{r^\tau}$ for $r \ge \frac{1}{\mu}$.

Next we note that the function $r \mapsto f_\mu(r) = 
-\log r  - 
\frac{1}{r^\tau}\frac{C_{\Omega,\tau}}{\mu^\tau}$ satisfies 
\begin{equation}
  \label{eq:boundary-conditions}
f_\mu(r)<0 \quad \text{for $r \ge 1$}\qquad \text{and}\qquad 
\lim_{r \to 0^+}f_\mu(r)= -\infty.
\end{equation}
Moreover, this function has two zeros $r_1(\mu), r_2(\mu)$ with $0<r_1(\mu)< \frac{1}{\mu} < r_2(\mu)<1$ and 
$$
f_\mu(r)\ge 0 \qquad \text{if and only if}\quad r_1(\mu) \le r \le r_2(\mu).
$$
To see this, we write 
$$
f_\mu(r)= \frac{1}{r^\tau}g(r^\tau) \qquad \text{with}\qquad g(s)=-\frac{s}{\tau} \log s - \frac{C_{\Omega,\tau}}{\mu^\tau}
$$
and note that $g$ is strictly concave since $s \mapsto g'(s)= -\frac{1}{\tau} - \log s$ is strictly decreasing. Consequently, $g$ has at most two zeros, and the same is true for $f$. Combining this with (\ref{eq:boundary-conditions}) and the fact that 
$$
f(1/\mu) = \log \mu  - C_{\Omega,\tau} > 0 
$$
since $\lambda \ge 2 C_{\Omega,\tau} >C_{\Omega,\tau}$ by assumption, the claim above follows. From (\ref{below2}), we thus obtain the lower bound
\begin{align}
  \label{eq:r-2-mu-est}
\int_{\Bbb R^{d}} &\Bigl(\lambda - \log |\xi| - 
\max \bigl\{ \frac{1}{|\xi|}, \frac{1}{|\xi|^\tau} \bigl\} C_{\Omega,\tau} \Bigr)_+d\xi \\
&\ge \mu^d\, \left|\Bbb S^{d-1} \right| \, \int_{\frac{1}{\mu}}^{r_2(\mu)} \left(-\log r  - 
\frac{1}{r^\tau}\frac{C_{\Omega,\tau}}{\mu^\tau}
\right)_+ \, r^{d-1}dr. \nonumber  
\end{align}
Next, we claim that
\begin{equation}
\label{r-2-mu-lower-bound}  
r_2(\mu) \ge r_3(\mu):= e^{\frac{1}{2\tau}\bigl(\sqrt{1- \frac{4\tau C_{\Omega,\tau}}{\mu^\tau}}\;-1\bigr)}.
\end{equation}
Here we note that $\frac{4\tau C_{\Omega,\tau}}{\mu^\tau}=\frac{4\tau C_{\Omega,\tau}}{e^{\tau \lambda}} <1$ since $\lambda \ge 2 C_{\Omega,\tau}$ by assumption. 
To see (\ref{r-2-mu-lower-bound}), we write 
$$
r_3(\mu)= e^{- c_\mu \frac{C_{\Omega,\tau}}{\mu^\tau}}\qquad \text{with}\qquad  
c_\mu = \frac{\mu^\tau}{2 \tau  C_{\Omega,\tau}}\Bigl(1 - \sqrt{1- \frac{4\tau C_{\Omega,\tau}}{\mu^\tau}}\Bigr),
$$
noting that 
$$
\frac{\tau C_{\Omega,\tau}}{\mu^\tau} c_\mu^2 -c_\mu +1= 0
$$
and therefore 
\begin{align*}
&f(r_3(\mu)) = f(e^{-\frac{c_\mu  C_{\Omega,\tau}}{\mu^\tau}})=\frac{c_\mu C_{\Omega,\tau}}{\mu^\tau} - 
\frac{1}{e^{-\tau \frac{c_\mu C_{\Omega,\tau}}{\mu^\tau}}}
\frac{C_{\Omega,\tau}}{\mu^\tau}\\
&=\frac{C_{\Omega,\tau}}{\mu^\tau e^{-\tau \frac{c_\mu C_{\Omega,\tau}}{\mu^\tau}}} \Bigl( c_\mu e^{-\tau \frac{c_\mu C_{\Omega,\tau}}{\mu^\tau}} - 1\Bigr) h\ge \frac{C_{\Omega,\tau}}{\mu^\tau e^{-\tau \frac{c_\mu C_{\Omega,\tau}}{\mu^\tau}}}\Bigl( c_\mu \bigl(1 - \tau \frac{c_\mu C_{\Omega,\tau}}{\mu^\tau}\bigr)-1\Bigr) = 0.
\end{align*}
This proves (\ref{r-2-mu-lower-bound}). As a consequence of the inequality $\sqrt{1-a} \ge 1-\frac{a}{2} -\frac{a^2}{2}$ for $0 \le a \le 1$, we also have 
$$
r_3(\mu) \ge e^{- \bigl(\frac{C_{\Omega,\tau}}{\mu^\tau}+\frac{4\tau C_{\Omega,\tau}^2}{\mu^{2\tau}}\bigr)} = : r_4(\mu).
$$
Consequently, 
\begin{align*}
\int_{\Bbb R^{d}} &\Bigl(\lambda - \log |\xi| - 
\max \bigl\{ \frac{1}{|\xi|}, \frac{1}{|\xi|^\tau} \bigl\} C_{\Omega,\tau} \Bigr)_+d\xi \\
&\ge \mu^d\, \left|\Bbb S^{d-1} \right| \, \int_{\frac{1}{\mu}}^{r_4(\mu)} \left(-\log r  - 
\frac{1}{r^\tau}\frac{C_{\Omega,\tau}}{\mu^\tau}
\right)_+ \, r^{d-1}dr\\  
&= 
  \mu^d\, \left|\Bbb S^{d-1} \right| \, \Bigl[-\frac{r^d}{d} \log r + \frac{1}{d^2}r^d - \frac{C_{\Omega,\tau}}{\mu^\tau(d-\tau)}r^{d-\tau}  \Bigr]_{\frac{1}{\mu}}^{r_4(\mu)}\\
&= 
  \mu^d\, \left|\Bbb S^{d-1} \right| \,\Bigl( \Bigl[-\frac{r_4(\mu)^d}{d} \log r_4(\mu) + \frac{1}{d^2}r_4(\mu)^d - \frac{C_{\Omega,\tau}}{\mu^\tau(d-\tau)}r_4(\mu)^{d-\tau}  \Bigr]\\
&- \Bigl[\frac{\mu^{-d}}{d} \log \mu + \frac{1}{d^2}\mu^{-d} - \frac{C_{\Omega,\tau}}{\mu^\tau(d-\tau)}\mu^{\tau-d}  \Bigr]\Bigr),
\end{align*}
which implies that 
\begin{align*}
\int_{\Bbb R^{d}} &\Bigl(\lambda - \log |\xi| - 
\max \bigl\{ \frac{1}{|\xi|}, \frac{1}{|\xi|^\tau} \bigl\} C_{\Omega,\tau} \Bigr)_+d\xi \\
&\ge
  \mu^d\, \left|\Bbb S^{d-1} \right| \,\Bigl(\frac{1}{d^2}r_4(\mu)^d - \frac{C_{\Omega,\tau}}{\mu^\tau(d-\tau)}r_4(\mu)^{d-\tau}- \frac{\mu^{-d}}{d} \log \mu - \frac{1}{d^2}\mu^{-d}  \Bigr)\\
&=
  \mu^d\, \left|\Bbb S^{d-1} \right| \,\Bigl(\frac{1}{d^2}e^{- d\bigl(\frac{C_{\Omega,\tau}}{\mu^\tau}+\frac{4\tau C_{\Omega,\tau}^2}{\mu^{2\tau}}\bigr)}
 - \frac{C_{\Omega,\tau}}{\mu^\tau(d-\tau)}e^{- (d-\tau)\bigl(\frac{C_{\Omega,\tau}}{\mu^\tau}+\frac{4\tau C_{\Omega,\tau}^2}{\mu^{2\tau}}\bigr)}\\
&- \frac{\mu^{-d}}{d} \log \mu - \frac{1}{d^2}\mu^{-d}  \Bigr).
\end{align*}
Since 
$$
e^{- d\bigl(\frac{C_{\Omega,\tau}}{\mu^\tau}+\frac{4\tau C_{\Omega,\tau}^2}{\mu^{2\tau}}\bigr)} \ge 1-  d\bigl(\frac{C_{\Omega,\tau}}{\mu^\tau}+\frac{4\tau C_{\Omega,\tau}^2}{\mu^{2\tau}}\bigr)
$$
and 
$$
e^{- (d-\tau)\bigl(\frac{C_{\Omega,\tau}}{\mu^\tau}+\frac{4\tau C_{\Omega,\tau}^2}{\mu^{2\tau}}\bigr)} \le 1,
$$
we conclude that 
\begin{align*}
\int_{\Bbb R^{d}} &\Bigl(\lambda - \log |\xi| - 
\max \bigl\{ \frac{1}{|\xi|}, \frac{1}{|\xi|^\tau} \bigl\} C_{\Omega,\tau} \Bigr)_+d\xi \\
&\ge 
  \mu^d\, \frac{\left|\Bbb S^{d-1} \right|}{d^2} \,\Bigl(1-  d\bigl(\frac{C_{\Omega,\tau}}{\mu^\tau}+\frac{4\tau C_{\Omega,\tau}^2}{\mu^{2\tau}}\bigr)
 - \frac{C_{\Omega,\tau}}{\mu^\tau(d-\tau)}- \mu^{-d}(d \log \mu + 1) \Bigr)\\
&= 
\frac{\left|B_d \right|}{d} \,\Bigl(\mu^d 
- C_{\Omega,\tau}(d-\frac{1}{d-\tau})\mu^{d-\tau} - 4d\tau C_{\Omega,\tau}^2 \mu^{d-2\tau} - (d \log \mu + 1) \Bigr)\\
&= 
\frac{\left|B_d \right|}{d} \,\Bigl(e^{d \lambda} 
- \frac{d(d-\tau)-1}{d-\tau} C_{\Omega,\tau}e^{(d-\tau)\lambda} - 4d\tau C_{\Omega,\tau}^2 
e^{(d-2\tau)\lambda} - (d \lambda + 1) \Bigr).
\end{align*}
Combining the last estimate with (\ref{intermediate-new}), we get the asserted lower bound.
\end{proof}


\section{Appendix: Note on a bound for Bessel functions}
\label{sec:appendix:-note-bound}

The following elementary bound might be known but seems hard to find in this form. 

\begin{lem}
For $\nu \ge \sqrt{3}-2$ and $0 \le x \le 2 \sqrt{2(\nu+2)}$ we have 
$$
|J_\nu(x)| \le \frac{x^\nu}{2^\nu \Gamma(\nu+1)}.
$$
\end{lem} 

\begin{proof}
We use the representation 
$$
J_\nu(x)= \Bigl(\frac{x}{2}\Bigr)^{\nu} \sum_{m=0}^\infty \frac{(-1)^m}{m! \Gamma(m+\nu + 1)} \Bigl(\frac{x}{2}\Bigr)^{2m}.
$$
For $0 \le x \le 2 \sqrt{2(\nu+2)}$ and $m \ge 1$, we have 
$$
\Bigl(\frac{x}{2}\Bigr)^{2} \le (m+1)(m+\nu+1) = \frac{(m+1) \Gamma(m+\nu + 2)}{\Gamma(m+\nu + 1)}
$$
and therefore 
\begin{equation}
  \label{eq:bessel-proof-1}
\frac{\Gamma(\nu+1)}{(m+1)! \Gamma(m+\nu + 2)} \Bigl(\frac{x}{2}\Bigr)^{2(m+1)} 
\le \frac{\Gamma(\nu+1)}{m! \Gamma(m+\nu + 1)} \Bigl(\frac{x}{2}\Bigr)^{2m}.
\end{equation}
Consequently, 
\begin{align*}
J_\nu(x) &= \frac{x^\nu}{2^\nu \Gamma(\nu+1)}\Bigl[1 + \sum_{m=1}^\infty \frac{(-1)^m \Gamma(\nu+1)}{m! \Gamma(m+\nu + 1)} \Bigl(\frac{x}{2}\Bigr)^{2m}\Bigr]\le \frac{x^\nu}{2^\nu \Gamma(\nu+1)}.
\end{align*}
From (\ref{eq:bessel-proof-1}) we also deduce that 
\begin{align*}
J_\nu(x) &\ge \frac{x^\nu}{2^\nu \Gamma(\nu+1)} \Bigl[1 - \frac{\Gamma(\nu+1)}{\Gamma(\nu + 2)} \Bigl(\frac{x}{2}\Bigr)^{2}+ \frac{\Gamma(\nu+1)}{2\Gamma(\nu + 3)} \Bigl(\frac{x}{2}\Bigr)^{4}- \frac{\Gamma(\nu+1)}{6\Gamma(\nu + 4)} \Bigl(\frac{x}{2}\Bigr)^{6}\Bigr]\\
&= \frac{x^\nu}{2^\nu \Gamma(\nu+1)} \Bigl[1-\frac{1}{\nu + 1}f\bigl(\bigl(\frac{x}{2}\bigr)^{2} \bigr)\Bigr]
\end{align*}
with $f: \R \to \R$ given by $f(t)= t - \frac{t^2}{2(\nu+2)}+ \frac{t^3}{6(\nu+2)(\nu+3)}$. Since 
$$
f'(t)= 1- \frac{t}{\nu+2} + \frac{t^2}{2(\nu+2)(\nu+3)}, \qquad \text{and}\qquad f''(t)= \frac{1}{\nu+2}\bigl(\frac{t}{\nu+3}- 1\bigr)
$$
we have 
$$
f'(t) \ge f'(\nu+3) = 1-  \frac{\nu+3 }{\nu+2} + \frac{\nu+3}{2(\nu+2)}
= 1 - \frac{1}{2}\frac{\nu+3 }{\nu+2} \ge 0 
\quad \text{for $t \in \R$ if $\nu \ge -1$}
$$
and therefore 
$$
f(t) \le f(2(\nu +2))= 2(\nu+2) - \frac{[2(\nu+2)]^2}{2(\nu+2)}+ \frac{[2(\nu+2)]^3}{6(\nu+2)(\nu+3)}= \frac{4(\nu+2)^2}{3(\nu+3)} 
$$
for $t \le 2(\nu+2)$ if $\nu \ge -1$. Since $\frac{4(\nu+2)^2}{3(\nu+3)} \le \frac{2}{\nu+1}$ for $\nu \ge \sqrt{3}-2$, we conclude that
$$
J_\nu(x) \ge \frac{x^\nu}{2^\nu \Gamma(\nu+1)} \Bigl[1-\frac{1}{\nu + 1}f\bigl(\bigl(\frac{x}{2}\bigr)^{2} \bigr)\Bigr]\ge - \frac{x^\nu}{2^\nu \Gamma(\nu+1)}.
$$
for $\nu \ge \sqrt{3}-2$ and $0 \le x \le 2 \sqrt{2(\nu+2)}$. The claim thus follows. 
\end{proof}

\noindent
{\it Acknowledgements}.
AL was supported by the RSF grant  19-71-30002.

\end{document}